\documentclass[leqno,12pt]{amsart}
\usepackage{amsfonts}
\usepackage{amsmath,amssymb,amsthm}
\usepackage{amsfonts}

\newtheorem{theorem}{Theorem}[section]
\newtheorem{proposition}[theorem]{Proposition}
\newtheorem{lemma}[theorem]{Lemma}
\newtheorem{corry}[theorem]{Corollary}

\theoremstyle{definition}
\newtheorem{example}[theorem]{Example}
\theoremstyle{definition}
\newtheorem{rem}[theorem]{Remark}

\numberwithin{claim}{theorem}

\renewenvironment{proof}{\textit{Proof.}}{\hfill\ensuremath{\qed}}

\def \qed{\hfill{\hbox{$\square$}}}

\numberwithin{equation}{section}

\begin{document}
\title[Surfaces in $\mathcal{LC}^n\times\mathbb R$]{Codimensions-Two Submanifolds Contained in the Light-like Hypercylinder $\mathcal{LC}^n\times\mathbb R$}

\author[A. Gineli]{Ali Gineli}
\address{Department of Mathematics, Faculty of Science and Letters, Istanbul Technical University, Istanbul, T{\"u}rkiye}
\email{gineli\_ali@hotmail.com}

\author[H. Y\"ur\"uk]{Hazal Y\"ur\"uk}
\address{Department of Mathematics, Faculty of Science and Letters, Istanbul Technical University, Istanbul, T{\"u}rkiye}
\email{yuruk22@itu.edu.tr}

\author[N. Cenk Turgay]{Nurettin Cenk Turgay}
\address{Department of Mathematics, Faculty of Science and Letters, Istanbul Technical University, Istanbul, T{\"u}rkiye}
\email{turgayn@itu.edu.tr}

\subjclass[2010]{53C42}
\keywords{Light cone, Minkowski space-time, space-like surfaces, degenerated hypersurfaces} 

\begin{abstract}
In this paper, we investigate space-like codimension-two submanifolds of    the Lorentz-Minkowski space $\mathbb{E}_1^{n+2}$ constrained to lie on the light-like  hypercylinder $\mathcal{LC}^n \times \mathbb{R}$ over the light cone $\mathcal{LC}^n$. By constructing a geometrically defined global frame field on the submanifold, we analyze the geometric interpretations of the associated extrinsic invariants by providing characterizations of (pseudo-)isotropic and pseudo-umbilical submanifolds, as well as submanifolds with flat normal bundle. In particular, we obtain local classification theorems for these classes of submanifolds.
\end{abstract}

\maketitle

\section{Introduction}
In recent years, there has been considerable interest in the study of Riemannian submanifolds contained within light-like hypersurfaces of the Lorentz-Minkowski space $\mathbb{E}_1^{n+2}$. A particularly rich setting arises when these submanifolds are constrained to lie on the light cone or its generalizations.  In particular, the degeneracy of the ambient hypersurface imposes constraints on the shape operators as well as on the Riemann and normal curvature tensors of the submanifold, which significantly affect its extrinsic geometry.

In \cite{Palomo_Romero_2013}, Riemannian surfaces in $\mathbb{E}_1^{4}$ contained in the light cone $\mathcal{LC}^3$ are analyzed, with detailed descriptions of their intrinsic and extrinsic geometry, where, \textit{throughout this paper}, the term light cone refers to the \textit{future component }of the light cone unless otherwise stated. The authors construct a geometrically defined global frame field on the submanifold by using a distinguished smooth function defined on any Riemannian surface contained in the light cone of four-dimensional Lorentz–Minkowski space, demonstrating how this function affects both intrinsic and extrinsic geometry.  Subsequently, in \cite{alias-etall2019}, Al\'ias \textit{et al.} extend these results to codimension-two Riemannian submanifolds of arbitrary dimension $n$ in $\mathbb{E}_1^{n+2}$ lying on the light cone $\mathcal{LC}^{n+1}$.

Another significant line of research deals with isometric immersions into the light cone of the Minkowski spaces and associated rigidity phenomena, \cite{LiuYu2018,Palmas_Palomo_Romero_2018,Palomo-etall2014}. In \cite{LiuYu2018}, authors obtain a rigidity result which proves that a Riemannian submanifold contained in the light cone of Minkowski space uniquely determines, up to a Lorentz transformation, a conformal immersion  into the unit sphere $\mathbb S^n$ of the Euclidean space $\mathbb E^{n+1}$. In \cite{Palomo-etall2014}, compact totally umbilical Riemannian surfaces in $\mathbb{E}_1^4$ lying within $\mathcal{LC}^3$ are characterized by the condition on their Gaussian curvature.

An important direction in this field concerns Riemannian submanifolds lying within light-like hypersurfaces of other Lorentzian manifolds, including non-flat Lorentzian space-forms. For example, trapped submanifolds in the light cone and in the past infinity of the steady-state region of de Sitter spaces are studied in \cite{alias-etall2018}, where authors investigate codimension-two compact Riemannian submanifolds lying in the light cone and proves that they are conformally diffeomorphic to $\mathbb S^n$. Further, characterization of compact marginally trapped submanifolds contained in these light-like hypersurfaces were provided, separately. On the other hand, the paper \cite{canovas-etall2021} considers compact totally umbilical surfaces through light cones of de Sitter space $\mathbb S^n_1$ and anti-de Sitter space $\mathbb H^n_1$. In particular, estimation of the first eigenvalue of the Laplace operator of the submanifolds on a light cone of $\mathbb S^n_1$ and  $\mathbb H^n_1$  were obtained. On the other hand, codimension-two Riemannian submanifolds of a Lorentzian manifold $(\widetilde {M}^{n+2}_1, \widetilde{g})$ immersed into a light-like hypersurface $\mathcal{L}$ are studied in \cite{Gutierrez_Olea_2021}, where authors obtain Ricci curvature, the second fundamental form and the shape operators of $M$  in terms of the  radical distribution, a screen distribution and corresponding light-like transverse vector field of $\mathcal{L}$. In particular, characterization theorem related with totally umbilical submanifolds are proven.

In this work, we focus on (pseudo-)isotropic and pseudo-umbilical Riemannian submanifolds, concepts that play a significant role in the geometry of submanifolds of Lorentzian ambient spaces. Following the framework established by B.-Y. Chen, the notion of pseudo-umbilical submanifolds was introduced as a natural generalization of umbilical hypersurfaces, \cite{ChenBook2011}. This concept has proven to be particularly significant in Lorentzian geometry, \cite{Chen2010, ChenDillen2002, DillenVerstraelen1990}. On the other hand, after the notion of pseudo-isotropic submanifolds was given by Young Ho Kim in \cite{YoungHoKim1995}, this class of submanifolds of Lorenzian manifolds have been studied in several works, \cite{Cabrerizo_Gomez_Fernandez_2008,Cabrerizo_Gomez_Fernandez_2009,Cabrerizo-etall2010}. For instance, Riemannian pseudo-isotropic and marginally trapped surfaces in the Minkowski space-time are investigated in \cite{Cabrerizo-etall2010}, where Cabrerizo \textit{et al.} refer to pseudo-isotropic submanifolds simply as isotropic submanifolds -- a terminology that we also adopt in this paper. In that work, the authors analyze the relationship between the isotropy function and the mean curvature vector field and obtain classification results for isotropic marginally trapped surfaces of Lorentzian space forms without flat points.

The aim of this paper is to study Riemannian submanifolds in the Minkowski space $\mathbb{E}_1^{n+2}$ lying on the  light-like  hypercylinder  $\mathcal{LC}^n \times \mathbb{R}$ over the light cone $\mathcal{LC}^n$. Section~2 presents the preliminaries, including the conventions, notations, and fundamental concepts from submanifold theory that will be used in the rest of the paper. In Section~3, we construct a geometrically defined global frame field on a codimension-two submanifold and investigate characterizations of extrinsic invariants associated with this frame. Section~4 is devoted to a local classification of pseudo-umbilical submanifolds of dimension $n > 2$. Finally, in Section~5, we study Riemannian surfaces lying on $\mathcal{LC}^2 \times \mathbb{R}$.

\section{Basic Notation, Concepts and Definitions}
In this section, we provide the basic notation used throughout the paper, along with a brief summary of the theory of submanifolds in semi-Riemannian space forms.

Let  $\mathbb E_1^k$  denote the $k$-dimensional Minkowski space with  the metric tensor $\tilde{g}=\langle,\rangle$  defined by
$$\tilde{g}=\langle \cdot ,\cdot\rangle=-dx_1^2+\sum\limits_{i=2}^{k}dx_i^2$$
and Levi-Civita connection $\widetilde\nabla$, where $(x_1,x_2,...,x_{k})$ is a Cartesian coordinate system on $\mathbb R^{k}.$ A vector $v \in \mathbb E_1^{k} $ is said to be space-like if $ \langle v,v\rangle>0$ or $v=0$, light-like if $\langle v,v\rangle=0,\ v\neq0$ and  time-like if $\langle v,v\rangle<0.$

Put $k=n+2$. Then, the (future component of the) light cone $\mathcal {LC}^{n+1}$ of the Minkowski space $\mathbb E^{n+2}_1$ is defined by 
$$\mathcal {LC}^{n+1}=\{x=(x_1,x_2,...,x_{n+1},x_{n+2})\in \mathbb E^{n+2}_1: \|x\|=0, \ x_1>0\}$$
and we are going to consider the light-like hypersurface $\mathcal {LC}^n\times\mathbb R$ of $\mathbb E^{n+2}_1$ defined by
\begin{equation}\label{LCnxRDef}
\mathcal {LC}^n\times\mathbb R=\{x=(x_1,x_2,...,x_{n+1},x_{n+2})\in \mathbb R^{n+2}: -x_1^2+x_2^2+...+x_{n+1}^2=0, \ x_1>0\},
\end{equation}
where we put $\|x\|=\sqrt{|\langle x,x\rangle|}$.

Let $M^n$ be an oriented Riemannian submanifold of $\mathbb E_1^k$ with Levi Civita connection $\nabla$, the second fundamental form $h$, shape operator $A$ and normal connection $\nabla^\bot$. Then, the Gauss and Weingarten formul\ae 
\begin{align}
\begin{split}
\nonumber\tilde{\nabla}_X Y&=\nabla_X Y+h(X,Y),\\
\nonumber\tilde{\nabla}_X \zeta&=-A_\zeta X+\nabla^\perp_X \zeta
\end{split}
\end{align}
are satisfied whenever $X$ and $Y$ are tangent to $M$ and $\zeta \in T^\perp M,$ where $T^\perp M$ denotes the normal bundle of $M.$ Note that $h$ and $A$ are related by
\begin{equation}\label{hArelby}
\langle h(X,Y),\zeta\rangle=\langle A_\zeta X,Y\rangle.
\end{equation}
Moreover, the Codazzi and Gauss equations are 
\begin{equation}\label{Codazzi}
(\overline{\nabla}{_X} h)(Y,Z)=(\overline{\nabla}_Y h)(X,Z)
\end{equation}
and 
\begin{equation}\label{GaussEq}
R(X,Y)Z=A_{h(Y,Z)}X-A_{h(X,Z)}Y,
\end{equation}
respectively, where the covariant derivative $(\overline{\nabla}_X h)(Y,Z)$ of the second fundamental form $h$  is defined by
\begin{align}
\nonumber(\overline{\nabla}_X h)(Y,Z)=\nabla_X ^\bot h(Y,Z)-h(\nabla_X Y,Z)-h(Y,\nabla_X Z).
\end{align}

On the other hand, the Ricci equation is
\begin{equation}\label{Ricci}
R^\perp(X,Y)\zeta=h(X,A_\zeta Y)-h(A_\zeta X,Y)
\end{equation}
whenever $X$, $Y$ are tangent and $\zeta$ is normal to $M$.  $M$ is said to have \textit{flat normal bundle} if $R^\perp=0$ and a normal vector field $\zeta$ is said to be \textit{parallel} if $\nabla^\perp_X \zeta = 0$ whenever $X$ is tangent to $M$. Since $M$ is a codimension-two submanifold, existence of a parallel normal vector field $\zeta\neq0$ is equivalent to having flat normal bundle because of the Ricci equation \eqref{Ricci}, \cite{ChenBook2011}.

Let $\{e_1, e_2, \dots, e_n\}$ be a local orthonormal frame field of the tangent bundle of $M$. We denote by $\omega_{ij}$ the connection 1-forms associated with this frame field, defined by $\omega_{ij}(e_k) = \langle \nabla_{e_k} e_i, e_j \rangle,$
which satisfy the skew-symmetry property $\omega_{ij} = -\omega_{ji}$. Throughout this paper, the indices $i,j,k,\ldots$ range from $1$ to $n$, whereas the indices $a,b,c,\ldots$ take values in $\{2, 3, \ldots, n\}$. On the other hand, a frame field $\{\theta,\xi\}$ of $T^\perp M$ is said to be pseudo-orthonormal if both of $\theta$ and $\xi$ are future-pointing and normalized light-like vectors, that is,
$$\langle\theta,\xi\rangle=-1\mbox{ and } \langle\theta,\partial_{x_1}\rangle <0.$$

The \textit{mean curvature vector field} $H$ of $M$ is defined by
$$
H = \frac{1}{n} \mathrm{trace}\, h = \frac{1}{n} \sum\limits_i h(e_i, e_i).
$$
The submanifold $M$ is said to be \textit{minimal} if $H = 0$, and \textit{marginally trapped} if $H$ is light-like at every point of $M$. Further, if there exists a smooth function $f_\zeta$ on $M$ such that
\begin{equation} \label{UmbWRTDef}
\langle A_\zeta X, Y \rangle = f_\zeta \langle X, Y \rangle,
\end{equation}
then $M$ is said to be \textit{umbilical} along the normal vector field $\zeta$. The submanifold $M$ is called \textit{pseudo-umbilical} if it is umbilical along $H$, and \textit{totally umbilical} if it is umbilical along every normal vector field $\zeta \in T^\perp M$, \cite{ChenBook2011}.

On the other hand, $M$ is said to be $\lambda$-isotropic for a function $\lambda\in C^\infty(M)$  if  the second fundamental form of $M$ satisfies
\begin{equation}\label{IsotropicDef}
\langle h(X, X), h(X, X)\rangle =\lambda
\end{equation}
for every \textit{unit} vector field $X$ tangent to $M$ (See \cite{Cabrerizo-etall2010,YoungHoKim1995}).

\subsection{Surfaces in the Minkowski Space-Time}
Now, consider the case $n=2, \ k=4$, that is, $M$ is a Riemannian surface in the Minkowski space-time $\mathbb E^4_1$, where we put  $x_1=t,x_2=x,x_3=y,x_4=z$. In this case, the Gaussian curvature of $M$ is defined by
$$
\nonumber K=R(e_1,e_2,e_2,e_1)
$$
and $M$ is said to be \textit{flat} if $K$ vanishes identically. Furthermore, the normal curvature of $M$ is defined by
$$K^\bot=\frac{\langle R^\bot(e_1,e_2)\zeta_1,\zeta_2 \rangle}{\langle\zeta_1,\zeta_1\rangle\langle\zeta_2\zeta_2\rangle-\langle\zeta_1,\zeta_2\rangle^2},$$
where $\{\zeta_1,\zeta_2\}$ is a positively oriented frame field for $T^\perp M.$

\begin{rem}
Throughout this paper, by congruency of two surfaces $M_1, M_2$ lying on $\mathcal {LC}^n\times\mathbb r\subset \mathbb E^{n+2_1}$ with position vectors $\phi_1, \phi_2$, we mean the existence of a rotation $\mathbf{Q}$ on the hyperplane $x_{n+2}=0$ ($z=0$ if $n=2$) of $\mathbb{E}^{n+2}_1$ (which is identified with $\mathbb{E}^{n+1}_1$) such that
$$\mathbf{Q} \circ \phi_1 = \phi_2.$$
\end{rem}

\subsection{Submanifolds Contained in Light Cone}
In this subsection, we are going to present a summary of basic facts on $(n-1)$-dimensional submanifolds of $\mathbb E^{n+1}_1$ contained in $\mathcal{LC}^n$.

First, let $n>2$ and $\gamma$ be the position vector of $\hat M$. Then, there exists a normal vector field $\eta $ such that   $\{\gamma,\eta\}$ is a pseudo-orthonormal frame field of the normal bundle of $\hat M$, that is
\begin{equation} \label{SubsectLCnEq1a}
\langle \gamma,\eta\rangle=-1,\qquad \langle \gamma,\gamma\rangle=\langle \eta,\eta\rangle=0.
\end{equation}
Note that $\gamma$ and $\eta$ are parallel along $\hat M$, i.e., 
\begin{equation} \label{SubsectLCnEq1b}
\hat\nabla^\perp\gamma=\hat\nabla^\perp\eta=0
\end{equation}
and $\gamma$ also satisfies
\begin{equation} \label{SubsectLCnEq1c}
\hat A_\gamma=-I,
\end{equation}
where we put $\hat A$ and $\hat\nabla^\perp$ for the shape operator and normal connection of $\hat M$, respectively and $I$ stands for the identity operator on $T\hat M$, \cite{alias-etall2019}. 
\begin{rem}\label{SubsectLCnRem1}
 Because of \eqref{SubsectLCnEq1c}, $\hat M$  is totally umbilical if and only if it is umbilical along $\eta$.
\end{rem}

On the other hand, throughout this paper, $\Sigma(a,\tau)$ will denote the totally umbilical submanifold of the Minkowski space $\mathbb E^{n+1}_1,\ n>2$ constructed in \cite{alias-etall2019}, where $\tau\in(0,\infty)$ and $0\neq a\in\mathbb E^{n+1}_1$.
\begin{example}\label{LCnExample1}\cite{alias-etall2019}
Let $0\neq a\in\mathbb E^{n+1}_1$ be a constant vector such that $\langle a,a\rangle=c\in\{\pm1,0\}$. Then, for a constant $\tau\in(0,\infty)$ the $(n-1)$-dimensional submanifold $\hat M=\Sigma(a,\tau)$ given by
$$\Sigma(a,\tau)=\{\gamma\in\mathcal{LC}^n:\langle \gamma,a\rangle=\tau\}$$
is a totally umbilical submanifold of $\mathbb E^{n+1}_1$. Furthermore, the vector fields  $\{\gamma,\eta\}$ normal to $\hat M^{n-1}$ given by \eqref{SubsectLCnEq1a} satisfies
\eqref{SubsectLCnEq1b}, \eqref{SubsectLCnEq1c} and
\begin{equation}\label{LCnExample1PVEq2}
\hat A_\eta=-\frac{c}{2\tau^2}I.
\end{equation}
\end{example}


Now, let $n=2$ and consider an arc length parametrized curve $\gamma$ contained in $\mathcal{LC}^2$. Then, similar to the previous case, there exists a vector field $\eta$ on $\gamma$ such that 
\begin{equation} \label{SubsectLCnEq2a}
\langle \gamma,\eta\rangle=-1,\qquad \langle \eta,\eta\rangle=\langle \gamma',\eta\rangle=0.
\end{equation}
Moreover, the Frenet-like equations
\begin{align} \label{SubsectLCnEq2b}
\begin{split}
\gamma''=&\kappa\gamma+\eta,\\
\eta'=&\kappa\gamma'\\
\end{split}
\end{align}
are satisfied for a function 
\begin{equation} \label{SubsectLCnEq2c}
\kappa:=\langle \gamma ', \eta '\rangle.
\end{equation}

\section{Extrinsic Properties of Submanifolds Lying on $\mathcal{LC}^n\times\mathbb R$}
In this section, we consider submanifolds  lying on  the light-like hypercylinder $\mathcal{LC}^n\times\mathbb R$ of the Minkowski space $\mathbb E^{n+2}_1$.

Let $M$ be an $n$-dimensional non-degenerated submanifolds of $\mathbb E^{n+2}_1$ with the position vector $\phi=(\phi_1,\phi_2,\hdots,\phi_{n+2})$. Assume that $M$ lies on  $\mathcal{LC}^n\times\mathbb R$, i.e.,
$$-\phi_1^2+\phi_2^2+\phi_3^2+\cdots+\phi_{n+1}^2=0,$$
where $\mathcal{LC}^n\times\mathbb R$ is the hypersurface of  $\mathbb E^{n+2}_1$ defined by  \eqref{LCnxRDef}.
Therefore, the vector field $\theta$ given by
\begin{equation}\label{LCnxRPhiDef}
\theta=(\phi_1,\phi_2,\hdots,\phi_{n+1},0)
\end{equation}
is normal to $M$ and it is light-like. Consequently, $M$ is Riemannian and there exists a vector field $\xi$ normal to $M$ such that
\begin{equation*}
\langle\xi,\xi\rangle=0,\quad\langle\xi,\theta\rangle=-1.
\end{equation*}
Note that we have
 \begin{equation}\label{LCnxRPhiDef2}
\theta=x-\langle x,\partial_{x_{n+2}}\rangle\partial_{x_{n+2}}.
\end{equation}

Next, in order to construct an orthonormal frame field for the tangent bundle of $M$, we consider the orthogonal decomposition
\begin{equation}\label{LCnxRxn2Decomp}
\left.\partial_{x_{n+2}}\right|_M=\left(\partial_{x_{n+2}}\right)^T+\left(\partial_{x_{n+2}}\right)^\perp
\end{equation}
of the vector field $\partial_{x_{n+2}}$, where $\left(\partial_{x_{n+2}}\right)^T$ and $\left(\partial_{x_{n+2}}\right)^\perp$ denote the tangential and normal part of $\partial_{x_{n+2}}$, respectively. Further, we define a smooth function $\alpha:M\to\mathbb R$ by
\begin{equation}\label{LCnxRalphaDef}
\nabla^\perp_{\left(\partial_{x_{n+2}}\right)^T}\theta=\alpha \theta.
\end{equation}

By taking into account that $\partial_{x_{n+2}}$ is parallel along $\mathbb E^{n+2}_1$ and $x$ is concircular, we get the covariant derivative of both sides of \eqref{LCnxRPhiDef2} along a vector field $X\in TM$ to get
\begin{eqnarray*}
\widetilde\nabla_X\theta=X-\langle X,\partial_{x_{n+2}}\rangle\partial_{x_{n+2}}.
\end{eqnarray*}
from which, along with \eqref{LCnxRxn2Decomp}, we have
\begin{subequations}\label{LCnxRxn2DecompDerEq1All}
\begin{eqnarray}
\label{LCnxRxn2DecompDerEq1a} A_\theta X&=&-X+\langle X,\left(\partial_{x_{n+2}}\right)^T\rangle\left(\partial_{x_{n+2}}\right)^T,\\
\label{LCnxRxn2DecompDerEq1b} \nabla^\perp_X\theta&=&-\langle X,\left(\partial_{x_{n+2}}\right)^T\rangle\left(\partial_{x_{n+2}}\right)^\perp.
\end{eqnarray}
\end{subequations}
Next, by combining \eqref{LCnxRxn2DecompDerEq1b} for $X=\left(\partial_{x_{n+2}}\right)^T$ with \eqref{LCnxRalphaDef}, we obtain 
\begin{equation}\label{LCnxRKeyEq1}
\alpha \theta=-\|\left(\partial_{x_{n+2}}\right)^T\|^2\left(\partial_{x_{n+2}}\right)^\perp.
\end{equation}
from which we observe that $\|\left(\partial_{x_{n+2}}\right)^\perp\|=0$ because being Riemannian of  $M$ implies $$\|\left(\partial_{x_{n+2}}\right)^T\|\neq0.$$ Therefore, we also have $\|\left(\partial_{x_{n+2}}\right)^T\|=1$. Thus, we put 
\begin{equation}\label{LCnxRe1Def}
e_1=\left(\partial_{x_{n+2}}\right)^T.
\end{equation}
Then, by combining \eqref{LCnxRKeyEq1} and \eqref{LCnxRe1Def} with \eqref{LCnxRxn2Decomp}, we obtain
\begin{equation}\label{LCnxRxn2DecompNew}
\left.\partial_{x_{n+2}}\right|_M=e_1-\alpha \theta.
\end{equation}

\begin{rem}\label{LCnxRxn2DecompNewRem}
Because of \eqref{LCnxRxn2DecompNew}, if $\alpha$ vanishes identically on an open, connected subset $\mathcal{M} \subset M$, then the vector field $\partial_{x_{n+2}}$ becomes tangent to $M$. In this case, we have
$$
\mathcal{M} = \hat{M} \times {I},
$$
for an open interval ${I}$ and a submanifold $\hat{M}^{n-1}$ lying in $\mathcal{LC}^n \subset \mathbb{E}_1^{n+1}$. Therefore, for the remainder of this paper, we assume that $\alpha$ does not vanish at any point of $M$.
\end{rem}

On the other hand, let $D$ be the distribution on $M$ defined by
\begin{equation}\label{LCnxRDistrDDef}
D=\mathrm{span\,} \{e_1\}
\end{equation}
with the orthonormal complementary $D^\perp$ in $TM$. Consider the linear, self-adjoint endomorphism $L:D^\perp\to D^\perp$ given by
$$L=\Pi\circ A_\xi,$$
where $\Pi:TM\to D^\perp$ is the canonical projection and let $\{e_2, e_3, \ldots, e_n\}$ be a set of orthonormal vectors such that
\begin{equation} \label{LCnxRDistreaDef}
Le_a = \beta_a e_a.
\end{equation}


Now, we are ready to prove the following lemma.
\begin{lemma}\label{LCnxRLemma1}
Let  $M$ be an $n$-dimensional submanifold of the Minkowski spaced $\mathbb E^{n+2}_1$ lying on  $\mathcal{LC}^n\times\mathbb R$. Then, there exists an orthonormal frame field $\{e_1,e_2,\hdots,e_n\}$ of the tangent bundle of $M$ and a pseudo-orthonormal frame field $\{\theta,\xi\}$ such that the following conditions hold.
\begin{itemize}
\item [(a)] $e_1$ and $\theta$ satisfy \eqref{LCnxRxn2DecompNew} for an $\alpha\in C^\infty(M)$,
\item [(b)] The Levi-Civita connection $\nabla$ of $ M$ satisfies
\begin{equation} \label{LCnxRMLCConn}
\nabla_{e_1}e_1=0,\qquad \nabla_{e_a}e_1=\alpha e_a,
\end{equation}
\item [(c)] The shape operators of $M$ have the matrix representation

\begin{eqnarray} 
\label{LCnxRMShOPP1a}A_\theta&=&\begin{pmatrix}
0&0&0&\cdots&0\\
0&-1&0&\cdots&0\\
0&0&-1&\cdots&0\\
\vdots&\vdots&\vdots&\ddots&0\\
0&0&0&\cdots&-1
\end{pmatrix}, \\
\label{LCnxRMShOPP1b} A_\xi&=&\begin{pmatrix}
-e_1(\alpha)-\alpha^2&-e_2(\alpha)&-e_3(\alpha)&\cdots&-e_n(\alpha)\\
-e_2(\alpha)&\beta_2&0&\cdots&0\\
-e_3(\alpha)&0&\beta_3&\cdots&0\\
\vdots&\vdots&\vdots&\ddots&0\\
-e_n(\alpha)&0&0&\cdots&\beta_n
\end{pmatrix},
\end{eqnarray} 
for some $\beta_2, \beta_3,\hdots, \beta_n\in C^\infty(M)$.
\item [(d)] The normal connection of $M$ has the form
\begin{eqnarray}
\label{LCnxRMNormConn1a} \nabla^\perp_{e_1}\theta=\alpha\theta,&\qquad& \nabla^\perp_{e_1}\xi=-\alpha\xi,\\
\label{LCnxRMNormConn1b} \nabla^\perp_{X}\theta=\nabla^\perp_{X}\xi=0, &&\mbox{whenever $X\in D^\perp$}.
\end{eqnarray}
\item [(e)] The second fundamental form of $M$ has the form
\begin{eqnarray}
\label{LCnxRM2ndFun1a} h(e_1,e_1)=\left(e_1(\alpha)+\alpha^2\right) \theta, &\qquad& h(e_1,e_a)=e_a(\alpha)\theta,\\
\label{LCnxRM2ndFun1b} h(e_a,e_a)=-\beta_a \theta+\xi, &\qquad& h(e_a,e_b)=0,\ a\neq b.
\end{eqnarray}
\end{itemize}
\end{lemma}

\begin{proof}
Let $e_2,\hdots,e_n$ be tangent vector fields defined by \eqref{LCnxRDistreaDef} and  $e_1,\xi$ be vector fields defined by  \eqref{LCnxRxn2DecompNew}. Then, by combining \eqref{LCnxRalphaDef} and \eqref{LCnxRDistreaDef}, we obtain \eqref{LCnxRMNormConn1a}. Moreover, \eqref{LCnxRxn2DecompDerEq1a} for $X=e_1$ implies
\begin{equation}\label{LCnxRLemma1Eq1}
A_\theta e_1=0
\end{equation}
and from \eqref{LCnxRxn2DecompDerEq1All} for $X=e_a$ we have
\begin{eqnarray}
\label{LCnxRLemma1Eq2} A_\theta e_a&=&-e_a
\end{eqnarray}
and \eqref{LCnxRMNormConn1b}. Therefore, we have  the part (d). Moreover, \eqref{LCnxRLemma1Eq1} and \eqref{LCnxRLemma1Eq2} imply \eqref{LCnxRMShOPP1a} from which, along with \eqref{LCnxRDistreaDef}, we obtain  \eqref{LCnxRM2ndFun1b}.

Next, we get the covariant derivative of both sides of \eqref{LCnxRxn2DecompNew} along $e_i$ to get
\begin{equation}\label{LCnxRxn2DecompNewCD}
0=\nabla_{e_i}e_1+h(e_1,e_i)-e_i(\alpha) \theta+\alpha A_\theta e_i-\alpha \nabla^\perp_{e_a}\theta.
\end{equation}
because $\partial_{x_{n+2}}$ is parallel along $\mathbb E^{n+2}_1$. The tangential and normal  the part of \eqref{LCnxRxn2DecompNewCD} imply
\begin{eqnarray}
\label{LCnxRxn2DecompNewC1} \nabla e_1&=&-\alpha A_\theta 
\end{eqnarray}
and
\begin{eqnarray}
\label{LCnxRxn2DecompNewCD2} h(e_1,e_i)=e_i(\alpha) \theta+\alpha \nabla^\perp_{e_i}\theta,
\end{eqnarray}
respectively. By combining \eqref{LCnxRxn2DecompNewC1} with \eqref{LCnxRMShOPP1a}, we get \eqref{LCnxRMLCConn} and \eqref{LCnxRxn2DecompNewCD2} implies \eqref{LCnxRM2ndFun1a}. So, we get the  part (b) and part (e) which also yields part (c) because of \eqref{hArelby}. Hence, the proof is completed.  
\end{proof}


We have the following corollaries of Lemma \ref{LCnxRLemma1}. 
\begin{corry}\label{LCnxRLemma1Corry1}
Let $M$ be an $n$-dimensional submanifold of the Minkowski spaced $\mathbb E^{n+2}_1$ lying on  $\mathcal{LC}^n\times\mathbb R$  and assume that $n>2$. Then, $M$ is pseudo-umbilical if and only if the functions $\alpha$ and $\beta_a$ appearing on Lemma \ref{LCnxRLemma1} satisfy
\begin{subequations} \label{LCnxRLemma1Corry1Eq1All}
\begin{eqnarray}
\label{LCnxRLemma1Corry1Eq1a} \beta_2=\beta_3=\cdots=\beta_n&=&\beta,\\
\label{LCnxRLemma1Corry1Eq1b} e_1(\alpha )+\alpha ^2&=&-\frac{2n-2}{n-2}\beta,\\
\label{LCnxRLemma1Corry1Eq1c} e_a(\alpha )&=&0.
\end{eqnarray}
\end{subequations}
\end{corry}
\begin{proof}
As a result of part (e) of Lemma \ref{LCnxRLemma1}, we obtain the mean curvature vector $H$ of $M$ as
\begin{equation}\label{LCnxREq1}
H=H_\theta \theta+\frac {n-1}n\xi,
\end{equation}
where we put $H_\theta=\frac{1}{n} \left( e_1(\alpha )+\alpha ^2- \sum\limits_a\beta_a\right)$.
By a direct computation, using part (c) of Lemma \ref{LCnxRLemma1} and \eqref{LCnxREq1}, we obtain the matrix representation of $A_H$ with respect to $\{e_1,e_2,\hdots,e_n\}$ as
$$A_H=\left(
\begin{array}{ccccc}
 \frac{-n+1}{n}\left(\alpha ^2+e_1(\alpha )\right) &  \frac{-n+1}{n}e_2(\alpha ) &  \frac{-n+1}{n}e_3(\alpha ) & \hdots &  \frac{-n+1}{n}e_n(\alpha ) \\
 \frac{-n+1}{n}e_2(\alpha ) & -H_\theta+\frac{n-1}{n}\beta_2 & 0 &\hdots & 0 \\
 \frac{-n+1}{n}e_3(\alpha ) & 0 & -H_\theta+\frac{n-1}{n}\beta_3 & \hdots & 0 \\
 \vdots & \vdots  &  \vdots  & \ddots  & \vdots \\
 \frac{-n+1}{n}e_n(\alpha ) & 0 & 0 & \hdots & -H_\theta+\frac{n-1}{n}\beta_n \\
\end{array}
\right).$$
 Therefore, \eqref{UmbWRTDef} is satisfied for $\zeta=H$ if and only if \eqref{LCnxRLemma1Corry1Eq1All} is satisfied.
\end{proof}

\begin{corry}\label{LCnxRLemma1Corry2}
Let $M$ be an $n$-dimensional submanifold of the Minkowski space $\mathbb E^{n+2}_1$ lying on  $\mathcal{LC}^n\times\mathbb R$. Then, $M$ has flat normal bundle if and only if the function $\alpha$ appearing on Lemma \ref{LCnxRLemma1} satisfies \eqref{LCnxRLemma1Corry1Eq1c}.
\end{corry}
\begin{proof}
The proof follows from Ricci equation \eqref{Ricci} and part (c) of Lemma \ref{LCnxRLemma1}. 
\end{proof}

\begin{corry}\label{L2RIsotSurfacesLemma1}
Let $M$ be an $n$-dimensional submanifold of the Minkowski space $\mathbb E^{n+2}_1$ lying on  $\mathcal{LC}^n\times\mathbb R$. Then, we have the followings\begin{itemize}
\item[$(1)$] If $M$ is $\lambda$-isotropic, then $M$ is 0-isotropic, i.e., $\lambda=0$.
\item[$(2)$] When $n=2$, $M$ is $\lambda$-isotropic if and only if
\begin{equation}\label{L2RIsotSurfacesLemma1beta=0}
\beta_2=0.
\end{equation}
\item[$(3)$] When $n>2$, $M$ is $\lambda$-isotropic if and only if it is marginally trapped and its mean curvature vector satisfies 
\begin{equation}\label{L2RSurfacesLemma3AH=0}
A_H=0.
\end{equation}
\item[$(4)$] When $n>2$, if $M$ is $\lambda$-isotropic, then it is pseudo-umbilical.
\end{itemize}
\end{corry}
\begin{proof}
By considering (e) of Lemma \ref{LCnxRLemma1}, we obtain
\begin{equation} \label{Isotroplambda=0}
\langle h(e_1,e_1),h(e_1,e_1)\rangle=0,\qquad \langle h(e_a,e_a),h(e_a,e_a)\rangle=2\beta_a.
\end{equation}
So, if \eqref{IsotropicDef} is satisfied for a $\lambda$, then $\lambda=0$. Consequently, we have (1). Note that \eqref{Isotroplambda=0} also implies (2).

Now, let $n>2$ and assume that $M$ is $\lambda$-isotropic. Note that \eqref{Isotroplambda=0} also implies $\beta_a=0$. By using (b), (d) and (e) of Lemma \ref{LCnxRLemma1}, we observe that  the Codazzi equation \eqref{Codazzi} for $X=Z=e_a, \ Y=e_1$ and  the Codazzi equation \eqref{Codazzi} for $X=Z=e_b, \ Y=e_a$ take the form
\begin{align*}
\begin{split}
e_ae_a(\alpha)=&-\alpha(e_1(\alpha)+\alpha^2)+\omega_{ab}(e_a)e_b(\alpha),\\
\alpha e_a(\alpha)=&0,
\end{split}
\end{align*}
respectively. Thus, we have \eqref{LCnxRLemma1Corry1Eq1All}. Therefore, Corollary \ref{LCnxRLemma1Corry1} implies (4).

On the other hand, \eqref{LCnxRLemma1Corry1Eq1All} and \eqref{Isotroplambda=0}  imply that $H=\frac 1n\xi$ and $A_\xi=0$. Therefore, we have the necessary condition of (3).  Conversely, if \eqref{L2RSurfacesLemma3AH=0} is satisfied, then the first normal space of $M$ will be the sub-bundle of $T^\perp M$ spanned by $\xi$ because of part (e) of Lemma \ref{LCnxRLemma1}. In this case, \eqref{IsotropicDef} is satisfied for $\lambda=0$.
\end{proof}


\section{Pseudo Umbilical Submanifolds}
In this section, we obtain local classification of pseudo-umbilical codimension-two submanifolds lying on $\mathcal{LC}^n\times\mathbb R$ for $n>2$.

Let $M$ be a submanifold lying on $\mathcal{LC}^n\times\mathbb R$.  We also assume that the function $\alpha$ does not vanish at any point of $M$ (See Remark \ref{LCnxRxn2DecompNewRem}). In the following lemma, we constructed a local coordinate system on  $M$ compatible with the global frame field $\{e_1,e_2,\hdots,e_n;\theta,\xi\}$ given in Lemma \ref{LCnxRLemma1}.
\begin{lemma}\label{LCnxRLemma2}
Let $\{e_1,e_2,\hdots,e_n;\theta,\xi\}$ be the global frame field on $M$ constructed in Lemma \ref{LCnxRLemma1} and $p\in M$. Then, there exists a local coordinate system  $\left(\mathcal N_p, \left(s,t_2,t_3,\hdots,t_n\right)\right)$ such that the following conditions hold:
\begin{itemize}
\item [(a)] $\mathcal N_p\ni p$,
\item [(b)] The vector field $e_1$ satisfies 
\begin{equation}
\label{SectLC2xRSurfsEq2a} \left.e_1\right|_{\mathcal N_p}=\frac{\partial}{\partial s},
\end{equation}
\item [(c)] The vector field $e_a$ satisfies $$\left.e_a\right|_{\mathcal N_p}\in\mathrm{span\,}\left\{\frac{\partial }{\partial t_2},\frac{\partial }{\partial t_3},\hdots,\frac{\partial }{\partial t_n}\right\},$$
\item [(d)] The parametrization of $\mathcal N_p$ has the form
\begin{equation}\label{LCnxRLemma2PosVect}
\phi(s,t)=(e^{\tilde\alpha(s,t)}\gamma_1(t),e^{\tilde\alpha(s,t)}\gamma_2(t),\hdots,e^{\tilde\alpha(s,t)}\gamma_{n+1}(t),s)
\end{equation}
for a smooth function $\tilde\alpha$ such that $e_1(\tilde\alpha)=\alpha$, where $\gamma=(\gamma_1,\gamma_2,\hdots,\gamma_{n+1})$ is the position vector of a space-like submanifold $\hat M^{n-1}$ lying on $\mathcal {LC}^n\subset\mathbb E^{n+1}_1$ and we put $t=(t_2,t_3,\hdots,t_n)$.
\end{itemize}
\end{lemma}

\begin{proof}
Let $D$ be the distribution on $M$ defined by \eqref{LCnxRDistrDDef} with the orthonormal complementary $D^\perp$ on $TM$. Note that we have
$$\langle[e_a,e_b],e_1\rangle=\langle\nabla_{e_a}e_b,e_1\rangle-\langle\nabla_{e_b}e_b,e_1\rangle=0$$
because \eqref{LCnxRMLCConn} implies that 
$$\langle\nabla_{e_a}e_b,e_1\rangle=-\langle e_b,\nabla_{e_a}e_1\rangle=0,\ a\neq b.$$
Therefore, $D^\perp$ is involutive , i.e., $[D^\perp,D^\perp]\subset D^\perp$. Moreover, $D$ is also involutive and we have
$$T_qM=D(q)\oplus D^\perp(q)$$
for all $q\in M$. Therefore, there exists a local orthonormal frame field $\left(\mathcal N_p, \left(s,t_2,t_3,\hdots,t_n\right)\right)$ such that
$$\left.D\right|_{\mathcal N_p}=\mathrm{span\,}\left\{\frac{\partial}{\partial s}\right\}\mbox{ and } \left.D^\perp\right|_{\mathcal N_p}=\mathrm{span\,}\left\{\frac{\partial }{\partial t_2},\frac{\partial }{\partial t_3},\hdots,\frac{\partial }{\partial t_n}\right\}$$
(See, for example, \cite[Lemma in p. 182]{KobayashiNomizu1963}). Consequently, we have part (a), (c) and
\begin{equation}\label{LCnxRLemma2Eq1}
\left.e_1\right|_{\mathcal N_p}=a_{11}\frac{\partial}{\partial s}
\end{equation}
for a smooth, non-vanishing function $a_{11}$ defined on $\mathcal N_p$.

Further, \eqref{LCnxRMLCConn} implies that
$$\langle [e_1,e_a],e_1\rangle =\langle \nabla_{e_1}e_a-\nabla_{e_a}e_1,e_1\rangle=-\langle \nabla_{e_1}e_1,e_a\rangle=0$$
from which we have
\begin{equation}\label{LCnxRLemma2Eq2}
[e_1,D^\perp]\subset D^\perp.
\end{equation}
Note that \eqref{LCnxRLemma2Eq1} implies
$$[e_1,\frac{\partial }{\partial t_a}]=\frac{\partial a_{11}}{\partial t_a} \frac{\partial}{\partial s}$$
 from which, along with \eqref{LCnxRLemma2Eq2}, we get $\frac{\partial a_{11}}{\partial t_a}=0$. Therefore, we have $a_{11}=a_{11}(s)$ on $\mathcal N_p$. Thus, by re-defining $s$ properly, one can assume that $a_{11}=1$ to get \eqref{SectLC2xRSurfsEq2a}. Therefore, we also have part (b).

Now, let $\phi(s,t)$ be the local parametrization of  $\mathcal N_p$, where, for simplicity, we put $t=(t_2,t_3,\hdots,t_n)$. Note that because of part (b) and part (c), \eqref{LCnxRxn2DecompNew} implies
\begin{equation}\label{LCnxRLemma2Eq3}
\left( \frac{\partial \phi_1}{\partial s},\frac{\partial \phi_2}{\partial s},\hdots,\frac{\partial \phi_{n+2}}{\partial s}\right)=\left(e_{11},e_{12},\hdots,e_{1n},1\right)
\end{equation}
and
\begin{equation}\label{LCnxRLemma2Eq4}
\left( \frac{\partial \phi_1}{\partial t_a},\frac{\partial \phi_2}{\partial t_a},\hdots,\frac{\partial \phi_{n+2}}{\partial t_a}\right)=\left(e_{a1},e_{a2},\hdots,e_{an},0\right),
\end{equation}
respectively. By combining \eqref{LCnxRLemma2Eq3} and \eqref{LCnxRLemma2Eq4}, we obtain
\begin{equation}\label{LCnxRLemma2Eq5}
\phi_{n+2}=s
\end{equation}
after a suitable translation on $s$. On the other hand, \eqref{LCnxRMNormConn1a} and \eqref{LCnxRMShOPP1a} imply
$$\frac{\partial \theta}{\partial s}=\widetilde\nabla_{e_1}\theta=\alpha \theta$$
from which, along with \eqref{LCnxRPhiDef}, we have 
$$\frac{\partial \phi_\beta}{\partial s}=\alpha\phi_\beta,\qquad \beta=1,2,\hdots,n+1$$
on $\mathcal N_p$. Therefore, $\phi_\beta$ has the form
\begin{equation}\label{LCnxRLemma2Eq6}
\phi_\beta=e^{\tilde\alpha(s,t)}\gamma_\beta(t),\qquad \beta=1,2,\hdots,n+1
\end{equation}
for some smooth functions $\tilde\alpha$ and $\gamma_\beta$ such that $e_1(\tilde\alpha)=\alpha$. 

Now, by combining \eqref{LCnxRLemma2Eq5} and \eqref{LCnxRLemma2Eq6}, we get \eqref{LCnxRLemma2PosVect} for an $\mathbb R^{n+1}$-valued smooth function $\gamma=(\gamma_1,\gamma_2,\hdots,\gamma_{n+1})$. Furthermore, by considering the coefficients of metric tensor of $\mathcal N_p$ we observe that 
 $\gamma$ is position vector of a space-like submanifold $\hat M^{n-1}$ of $\mathbb E^{n+1}_1$. Moreover, since $M$ lies on $\mathcal {LC}^n\times \mathbb R$, we have $\langle\gamma,\gamma\rangle=0$. Hence, we have part (d).
\end{proof}


Next, by considering Lemma \ref{LCnxRLemma2} and Corollary \ref{LCnxRLemma1Corry2}, we obtain following the proposition:
\begin{proposition}\label{LCnxRProp1}
Let $M$ be a submanifold lying on $\mathcal{LC}^n\times\mathbb R$. Then, $M$ has flat normal bundle if and only for all $p\in M$ there exists a neighborhood of $M$ which can be parametrized by 
\begin{equation}\label{LCnxRProp1PosVect}
\phi(s,t)=(\hat\alpha(s)\gamma_1(t),\hat\alpha(s)\gamma_2(t),\hdots,\hat\alpha(s)\gamma_{n+1}(t),s)
\end{equation}
for a smooth non-vanishing function $\hat\alpha$, where  $\gamma=(\gamma_1,\gamma_2,\hdots,\gamma_{n+1})$ is the position vector of a space-like submanifold $\hat M^{n-1}$ lying on $\mathcal {LC}^n\subset\mathbb E^{n+1}_1$.
\end{proposition}

\begin{proof}
In order to prove the necessary condition, assume that  $M$ has flat normal bundle, $p\in M$ and let $\left(\mathcal N_p, \left(s,t_2,t_3,\hdots,t_n\right)\right)$ be a local coordinate system satisfying conditions given in Lemma \ref{LCnxRLemma2}. For simplicity, we are going to put 
$$e^{\tilde\alpha}=\hat \alpha.$$
By Corollary \ref{LCnxRLemma1Corry2}, we have  \eqref{LCnxRLemma1Corry1Eq1c} which implies 
$$\hat \alpha(s,t)=\hat \alpha(s)$$
on $\mathcal N_p$ because of part (c) of Lemma \ref{LCnxRLemma2}. Hence, \eqref{LCnxRLemma2PosVect} turns into \eqref{LCnxRProp1PosVect}.

Converse follows from a direct computation and Corollary \ref{LCnxRLemma1Corry2}. 
\end{proof}


\eqref{LCnxRMShOPP1a} clearly implies that there are no  totally umbilical codimension-two submanifold of $\mathbb E^{n+2}_1$  lying on the hypercylinder $\mathcal{LC}^n\times\mathbb R$. However, there  exist totally umbilical submanifolds of $\mathbb E^{n+1}_1$ contained in the light cone  $\mathcal{LC}^{n}$ (See Example \ref{LCnExample1}). In fact, the local classification of totally umbilical submanifolds was obtained in \cite[Theorem 4.2]{alias-etall2019}:
\begin{theorem}\label{LCnThm1}\cite{alias-etall2019}
Let $\hat M$ be a space-like, totally umbilical codimension-two submanifold of $\mathbb E^{n+1}_1$ contained in $\mathcal{LC}^n$. Then, there exist an $a\in\mathbb E^{n+1}_1$ such that $\langle a,a\rangle=c\in\{\pm1,0\}$  and a  constant $\tau\in(0,\infty)$ such that
\begin{equation}\label{LCnThm1Eq1}
\hat M\subset \Sigma(a,\tau).
\end{equation}
\end{theorem}


Now, we are ready to prove the main result of this section. Note that in the proof of the following theorem, through a misuse of notation, we are going to put
$$\frac{\partial}{\partial t_a}=\phi_*\left(\frac{\partial}{\partial t_a}\right)=\gamma_*\left(\frac{\partial}{\partial t_a}\right).$$
\begin{theorem}\label{LCnxRProp2}
Let $M$ be a submanifold lying on $\mathcal{LC}^n \times \mathbb{R}$ with $n > 2$. Then, $M$ is a pseudo-umbilical submanifold of $\mathbb{E}_1^{n+2}$ if and only if there exists a neighborhood of $M$ that can be parametrized as in Proposition~\ref{LCnxRProp1}, for some function $\hat{\alpha}$, such that the associated space-like submanifold $\hat{M}^{n-1} \subset \Sigma(a, \tau)\subset \mathbb{E}_1^{n+1}$ is totally umbilical, and $\hat{\alpha}$ satisfies
\begin{equation}\label{LCnxRThm1Eq1}
(n-2)\tau^2 \hat{\alpha} \, \hat{\alpha}'' - (n-1)\left( c + \tau^2 (\hat{\alpha}')^2 \right) = 0.
\end{equation}
\end{theorem}

\begin{proof}
Similar to proof of Proposition \ref{LCnxRProp1}, we assume that  $p\in M$ and  consider a local coordinate system $\left(\mathcal N_p, \left(s,t_2,t_3,\hdots,t_n\right)\right)$ satisfying conditions given in Lemma \ref{LCnxRLemma2}. In order to prove the necessary condition, we assume that $M$ is pseudo-umbilical. Then, by Corollary \ref{LCnxRLemma1Corry1}, we have  \eqref{LCnxRLemma1Corry1Eq1All} for a function $\beta$ and $M$ has flat normal bundle. Therefore, Proposition  \ref{LCnxRProp1} implies that $\mathcal N_p$ has a parametrization given by \eqref{LCnxRProp1PosVect} for a function $\hat \alpha$ and  a space-like submanifold $\hat M^{n-1}$. Note that we have 
\begin{equation}\label{LCnxRThm1ProofEq0}
\alpha=\frac{\hat\alpha'}{\hat\alpha}
\end{equation}

Now, let $\{\gamma,\eta\}$ be the pseudo-orthonormal frame field for the normal bundle of $\hat M^{n-1}$ such that \eqref{SubsectLCnEq1b} is satisfied. Then, \eqref{LCnxRxn2DecompNew} and \eqref{LCnxRProp1PosVect} imply
\begin{eqnarray}
\label{LCnxRProp2Eq1b} \frac{\partial}{\partial t_a}&=& \left(\hat\alpha \frac{\partial }{\partial t_a},0 \right),\\
\label{LCnxRProp2Eq1d} \xi&=& \left(\frac{1}{\hat\alpha}\eta+\frac{\hat\alpha'{}^2}{2\hat\alpha}\gamma,\frac{\hat\alpha'}{\hat\alpha} \right).
\end{eqnarray}
By a direct computation using \eqref{LCnxRMShOPP1b}, \eqref{LCnxRMNormConn1b}, \eqref{LCnxRLemma1Corry1Eq1All} and \eqref{LCnxRProp2Eq1b} we obtain
\begin{equation}\label{LCnxRProp2Eq2} 
\widetilde\nabla_{\frac{\partial}{\partial t_a}}\xi=\left(-\beta\hat\alpha \frac{\partial }{\partial t_a},0 \right)
\end{equation}
On the other hand, from \eqref{SubsectLCnEq1b}, and \eqref{LCnxRProp2Eq1d}  we also have
\begin{equation}\label{LCnxRProp2Eq2a} 
\widetilde\nabla_{\frac{\partial}{\partial t_a}}\xi=\left(-\frac{1}{\hat\alpha}\hat A_\eta\frac{\partial}{\partial t_a}-\frac{\hat\alpha'{}^2}{2\hat\alpha}\hat A_\gamma\frac{\partial}{\partial t_a},0\right).
\end{equation}
By combining \eqref{LCnxRProp2Eq2a}  with  \eqref{LCnxRProp2Eq2} and using \eqref{SubsectLCnEq1c}, we get
\begin{equation}\label{LCnxRProp2Eq3}
\hat A_\eta\frac{\partial}{\partial t_a}=\left(\frac{\hat\alpha'{}^2}{2}+\beta\hat\alpha^2\right)\frac{\partial}{\partial t_a}.
\end{equation}
Therefore, because of Remark \ref{SubsectLCnRem1}, $\hat M$ is a totally umbilical submanifold of $\mathbb E^{n+1}_1$. Consequently, Theorem \ref{LCnThm1} implies \eqref{LCnThm1Eq1} for some $a,c,\tau$  as given in Example \ref{LCnExample1} and we have  \eqref{LCnExample1PVEq2}.

Next, by using \eqref{LCnExample1PVEq2} and \eqref{LCnxRProp2Eq3}, we obtain
\begin{equation}\label{LCnxRProp2Eq4}
\beta=-\frac{c+\tau ^2 (\hat\alpha')^2}{2\tau^2\hat\alpha^2}.
\end{equation}
Finally, by combining \eqref{LCnxRProp2Eq4}  with \eqref{LCnxRLemma1Corry1Eq1b}, we obtain \eqref{LCnxRThm1Eq1}. Hence, the proof of the necessary condition completed.

Converse follows from a direct computation by considering Corollary \ref{LCnxRLemma1Corry1} and Example \ref{LCnExample1}. 
\end{proof}


\subsection{Isotropic Submanifolds} In this subsection, as an application of Theorem \ref{LCnxRProp2}, we obtain local classification of isotropic submanifolds.

Let $M$ be a  submanifold of $\mathbb E^{n+2}_1$ lying on $\mathcal{LC}^n\times\mathbb R$, $n>2$ and $p\in M$. Note that Corollary \ref{L2RIsotSurfacesLemma1} implies that $M$ is $\lambda$-isotropic if and only if it is a pseudo-umbilical  with $e_1(\alpha)+\alpha^2=0$.

Now, assume that $M$ is a  $\lambda$-isotropic submanifold of $\mathbb E^{n+2}_1$. Then, since $M$ is  pseudo-umbilical, Theorem \ref{LCnxRProp2} implies that  there exists a neighborhood $\mathcal N_p$ of $M$ which can be parametrized by \eqref{LCnxRProp1PosVect} as given by Proposition  \ref{LCnxRProp1} for a function $\hat \alpha$ satisfying \eqref{LCnxRThm1Eq1} and a submanifold $\hat M^{n-1}\subset \Sigma(a,\tau)$ for some $a,c,\tau$ as given in Example \ref{LCnExample1}.

Note that $\alpha$ satisfies $\alpha(s,t)=\alpha(s)$ and \eqref{LCnxRThm1ProofEq0} on $\mathcal N_p$. Therefore, $e_1(\alpha)+\alpha^2=0$ and \eqref{LCnxRThm1ProofEq0} imply 
$$\hat \alpha=c_1s+c_2$$
 for some constants $c_1\neq0$ and $c_2$. Note that from \eqref{LCnxRThm1Eq1} we have $$c+\tau ^2 c_1^2=0$$
which yields that $c=-1$ and $c_1=\frac{\epsilon}{\tau}$ for an $\epsilon=\pm1$. Therefore, \eqref{LCnxRProp1PosVect} turns into
\begin{equation}\label{LCnxRIsotThmFORn>2PosVect}
\phi(s,t)=\left(\left(\frac{\epsilon s}{\tau} + c_2\right)\gamma_1(t), \left(\frac{\epsilon s}{\tau} + c_2\right)\gamma_2(t), \ldots, \left(\frac{\epsilon s}{\tau} + c_2\right)\gamma_{n+1}(t), s\right).
\end{equation}
Now, since $c=\langle a, a \rangle = -1$, up to a rotation on the hyperplane $x_{n+2}=0$, one may choose $a=(1,0,0,\hdots,0)\in \mathbb{E}_1^{n+1}$. In this case, $\Sigma(a, \tau)$ can be  parametrized by
$$(\gamma_1,\gamma_2,\hdots,\gamma_{n+1})=\left(\tau,\tau\Theta_1,\tau\Theta_2,\hdots,\tau\Theta_{n} \right),$$
where $(\Theta_1,\Theta_2,\hdots,\Theta_{n})$ is the position vector of the hypersphere $\mathbb S^{n-1}$ of $\mathbb E^{n}$. Hence, \eqref{LCnxRIsotThmFORn>2PosVect} turns into
\begin{equation}\label{LCnxRIsotThmFORn>2PosVectv2}
\phi(s,t)=\left(\left(\epsilon s + c_0\right), \left(\epsilon s + c_0\right)\Theta_1(t),\left(\epsilon s + c_0\right)\Theta_2(t), \ldots, \left(\epsilon s + c_0\right)\Theta_{n}(t), s\right)
\end{equation}
and we have proven the following theorem.
\begin{theorem}\label{LCnxRIsotThmFORn>2}
Let $M$ be a submanifold lying on $\mathcal{LC}^n \times \mathbb{R}$ with $n > 2$. Then, $M$ is a $\lambda$-isotropic submanifold of $\mathbb{E}_1^{n+2}$ if and only if there exists a neighborhood of $M$ which is congruent to the submanifold parametrized by \eqref{LCnxRIsotThmFORn>2PosVectv2}
for a constant  $c_0 \in \mathbb{R}$, where we put $t = (t_2, t_3, \ldots, t_n)$ and   $(\Theta_1,\Theta_2,\hdots,\Theta_{n})$ is the position vector of the unit hypersphere $\mathbb S^{n-1}$ of $\mathbb E^{n}$.
\end{theorem}

\section{Surfaces in $\mathcal{LC}^2\times\mathbb R$}
In this section, we study space-like surfaces in the Minkowski space-time $\mathbb E^4_1$ that lie on $\mathcal{LC}^2\times\mathbb R$.

Let $M$ be a Riemannian surface in $\mathbb E^4_1$, with the position vector $\phi$. If $M$ is assumed to lie on $\mathcal{LC}^2\times\mathbb R$, then, as a consequence of Lemma \ref{LCnxRLemma1} for $n=2$, there exists a frame field $\{e_1,e_2; \theta, \xi\}$ on $M$ such that $e_1$ and $\theta$ satisfy \eqref{LCnxRxn2DecompNew} for some $\alpha \in C^\infty(M)$ and 
\begin{subequations}\label{SectLC2xRSurfsEq1All}
\begin{eqnarray}
\label{SectLC2xRSurfsEq1a} \widetilde\nabla_{e_1}e_1=(e_1(\alpha)+\alpha^2)\theta, 	&\qquad& \widetilde\nabla_{e_2}e_1=\alpha e_2+e_2(\alpha)\theta,\\
\label{SectLC2xRSurfsEq1b} \widetilde\nabla_{e_1}e_2=e_2(\alpha)\theta,       		  &\qquad& \widetilde\nabla_{e_2}e_2=-\alpha e_1-\beta \theta+\xi,\\
\label{SectLC2xRSurfsEq1c} \widetilde\nabla_{e_1}\theta=\alpha\theta,              	&\qquad& \widetilde\nabla_{e_2}\theta=e_2,\\
\label{SectLC2xRSurfsEq1d} \widetilde\nabla_{e_1}\xi=(e_1(\alpha)+\alpha^2)e_1+e_2(\alpha)e_1-\alpha\xi, &\qquad& \widetilde\nabla_{e_2}\xi=e_2(\alpha)e_1-\beta e_2
\end{eqnarray}
\end{subequations}
for some smooth functions $\alpha$ and $\beta:=\beta_2$.


On the other hand, by Lemma \ref{LCnxRLemma2} for $n=2$, for all $p\in M$ there exists a local coordinate system  $\left(\mathcal N_p, \left(s,t\right)\right)$ such that the conditions \eqref{SectLC2xRSurfsEq2a} and
\begin{subequations}\label{SectLC2xRSurfsEq2All}
\begin{eqnarray}
\label{SectLC2xRSurfsEq2b} \left.e_2\right|_{\mathcal N_p}&=&\frac 1{\hat\alpha} \frac{\partial }{\partial t}, \\
\label{SectLC2xRSurfsEq2c} \phi(s,t)&=&(\hat\alpha(s,t)\gamma_1(t),\hat\alpha(s,t)\gamma_2(t),\hat\alpha(s,t)\gamma_3(t),s)
\end{eqnarray}
are satisfied, where we put $t_2=t$ and  $\gamma=(\gamma_1,\gamma_2,\gamma_3)$ is an \textit{arc-length parametrized} curve such that $\langle\gamma,\gamma\rangle=0$. Note that the functions $\alpha$ and $\hat\alpha$ are related by
\begin{eqnarray}
\label{SectLC2xRSurfsEq2d} \alpha=\frac{e_1(\hat\alpha)}{\hat\alpha}.
\end{eqnarray}
\end{subequations}

\subsection{Flat Surfaces}
In this subsection, first, we construct a flat surface by the following proposition. 
\begin{proposition}\label{SecLc2xRProp1}
Let $M$ be a ruled surface in $ \mathbb E^4_1 $ parametrized by 
\begin{equation}\label{SecLc2xRProp1Eq1}
\phi(s,t)=\left((s+a(t))\gamma_1(t),(s+a(t))\gamma_2(t),(s+a(t))\gamma_3(t),s\right)
\end{equation}
for a smooth function $a(v)$, where $\gamma=(\gamma_1,\gamma_2,\gamma_3)$ is an arc-length parametrized space-like curve such that $\langle\gamma,\gamma\rangle=0$. Then, $M$ lies on  $\mathcal{LC}^2 \times \mathbb{R} $  and it is flat. 
\end{proposition}
\begin{proof}
By a direct computation, we see that the induced metric tensor of $M$ is $$g=ds^2+(s+a(t)^2)dt^2$$ which yields that $M$ is flat.
\end{proof}


Next, we state the following corollary.
\begin{corry}\label{L2RSurfacesLemma0Flat}
A surface of the Minkowski space-time $\mathbb E^{4}_1$ lying on  $\mathcal{LC}^2\times\mathbb R$ is flat if and only if the function $\alpha$ appearing on \eqref{SectLC2xRSurfsEq1All} satisfies
\begin{equation}\label{SecLc2xRThm1Eq1}
e_1(\alpha)+\alpha^2=0.
\end{equation}
\end{corry}

\begin{proof}
Proof follows from the Gauss equation \eqref{GaussEq}, \eqref{SectLC2xRSurfsEq1c} and \eqref{SectLC2xRSurfsEq1d}.
\end{proof}


In the following theorem, a local classification of flat surfaces lying on $\mathcal{LC}^2 \times \mathbb{R}$ is given.
\begin{theorem}\label{SecLc2xRThm1}
Let $M$ be a surface  lying on $\mathcal{LC}^2\times \mathbb R$. Then $M$ is flat if and only if for all $p\in M$ there exists a neighborhood of $M$ which can be parametrized as given in Proposition \ref{SecLc2xRProp1}.
\end{theorem}

\begin{proof} 
Let $M$ be a surface in $\mathcal{LC}^2\times \mathbb R$, $p\in M$. Consider the frame field $\{e_1,e_2;\theta,\xi\}$ on $M$ defined by \eqref{LCnxRxn2DecompNew} for an $\alpha\in C^\infty(M)$ and the local coordinate system $\left(\mathcal N_p, \left(s,t\right)\right)$ satisfying the conditions \eqref{SectLC2xRSurfsEq2a} and \eqref{SectLC2xRSurfsEq2All}. Note that the equations appearing in \eqref{SectLC2xRSurfsEq1All} are also satisfied.

Because of Corollary \ref{L2RSurfacesLemma0Flat}, if $M$ is flat, then $\alpha$ satisfies \eqref{SecLc2xRThm1Eq1} which is equivalent to 
$$\alpha=\frac{1}{s+a(t)}$$
on $\mathcal N_p$ for a smooth function $a$ and \eqref{SectLC2xRSurfsEq2d} implies that
\begin{equation}\label{SecLc2xRThm1ProofEq1}
\hat\alpha(s,t)=b(t)(s+a(t))
\end{equation}
for a smooth, non-vanishing function $b$. However, by re-defining the curve $\beta$ and the parameter $t$ appropriately, one may put $b(t)=1$. Therefore, by combining \eqref{SecLc2xRThm1ProofEq1} with \eqref{SectLC2xRSurfsEq2c}, we obtain \eqref{SecLc2xRProp1Eq1}. So, we have shown the necessary condition. Sufficient condition is proven in Proposition \ref{SecLc2xRProp1}. Hence, the proof is completed.
\end{proof}

Next, we state the following corollary of Proposition \ref{LCnxRProp1} for $n=2$.
\begin{corry}\label{SecLc2xRCorryFNB}
Let $M$ be a surface  lying on $\mathcal{LC}^2\times \mathbb R$. Then, $M$ has flat normal bundle if and only for all $p\in M$ there exists a neighborhood of $M$ which can be parametrized by 
$$\phi(s,t)=\left(\hat\alpha(s)\gamma_1(t),\hat\alpha(s)\gamma_2(t),\hat\alpha(s)\gamma_3(t),s\right)$$
for a smooth non-vanishing function $\hat\alpha$, where $\gamma=(\gamma_1,\gamma_2,\gamma_3)$ is an arc-length parametrized space-like curve such that $\langle\gamma,\gamma\rangle=0$.
\end{corry}


\subsection{Isotropics Surfaces}
In this subsection, we consider $\lambda$-isotropic surfaces in $\mathbb E^4_1$ lying on   $\mathcal{LC}^2\times \mathbb R$.

Let $M$ be a space-like surfaces in $\mathbb E^4_1$ lying on  $\mathcal{LC}^2\times \mathbb R$, $p\in M$ and consider the orthonormal frame field $\{e_1,e_2;\theta,\xi\}$given by  \eqref{LCnxRxn2DecompNew}  and \eqref{SectLC2xRSurfsEq1All}and the local coordinate system  $\left(\mathcal N_p, \left(s,t\right)\right)$ satisfying the conditions \eqref{SectLC2xRSurfsEq2a} and \eqref{SectLC2xRSurfsEq2All}. In this case, \eqref{LCnxRxn2DecompNew} and \eqref{SectLC2xRSurfsEq2c} imply
\begin{eqnarray}
\label{SectLC2xRSurfsFrameEq1} \xi&=& \left(\frac{\hat\alpha_t^2+\hat\alpha^2\hat\alpha_s^2}{2\hat\alpha^3}\gamma+\frac{\hat\alpha_t}{\hat\alpha^2}\gamma'+\frac{1}{\hat\alpha(s,t)}\eta,\frac{\hat\alpha_s}{\hat\alpha} \right),
\end{eqnarray}
where indices $s$ and $t$ denote the corresponding partial derivatives and $\eta$ is the vector field on $\gamma$ defined by \eqref{SubsectLCnEq2a}. By a direct computation using \eqref{SubsectLCnEq2b}, \eqref{SectLC2xRSurfsEq2b} and \eqref{SectLC2xRSurfsFrameEq1}, we obtained
\begin{align}\label{L2RIsotSurfacesThm1LastEq}
\begin{split}
\widetilde\nabla_{e_2}\xi=&\frac{\hat\alpha  \hat\alpha_{st}-\hat\alpha_{t} \hat\alpha_{s}}{\hat\alpha ^3}e_1
+\frac{\hat\alpha ^2 \left(\hat\alpha_{s}^2+2 \kappa \right)+2 \hat\alpha\hat\alpha_{tt} -3 \hat\alpha_{t}^2}{2 \hat\alpha ^4}e_2,
\end{split}
\end{align}
where $\kappa$ is the function defined by \eqref{SubsectLCnEq2c}.

Now, assume that $M$ is $\lambda$-isotropic for a smooth function $\lambda$. Then, Corollary \ref{L2RIsotSurfacesLemma1} implies $\lambda=0$ and \eqref{L2RIsotSurfacesLemma1beta=0}. Therefore,  we have $\beta=0$ from which, along with \eqref{L2RIsotSurfacesThm1LastEq}, we obtain
\begin{equation}\label{SectLC2xRSurfsFrameEq2}
\hat\alpha ^2 \left(\hat\alpha_{s}^2+2 \kappa \right)+2 \hat\alpha\hat\alpha_{tt} -3 \hat\alpha_{t}^2=0.
\end{equation}
Hence, we have proven the following result:
\begin{theorem}\label{L2RIsotSurfacesThm1}
Let $M$ be a surface  lying on $\mathcal{LC}^2\times \mathbb R.$ Then,  $M$ is $\lambda$-isotropic if and only if for all $p\in M$ there exists a neighborhood of $M$ which can be parametrized by \eqref{SectLC2xRSurfsEq2c} for an $\hat\alpha$ satisfying \eqref{SectLC2xRSurfsFrameEq2}, where $\kappa$ is the function defined by \eqref{SubsectLCnEq2c}. 
\end{theorem}


\subsection{Pseudo-Umbilical Surfaces}
In this subsection, we are going to consider pseudo-umbilical surfaces lying on $\mathcal{LC}^2\times \mathbb R.$

\begin{proposition}\label{L2RSurfacesLemma3}
Let $M$ be a surface in the Minkowski space-time $\mathbb E^4_1$ lying on $\mathcal{LC}^2\times \mathbb R.$ Then, the following are equivalent to each other:
\begin{itemize}
\item[$(1)$] $M$ is pseudo umbilical.

\item[$(2)$] $\alpha$ and $\beta$ appearing in \eqref{SectLC2xRSurfsEq1All} satisfies
\begin{equation}\label{L2RSurfacesLemma3Equa1}
e_1(\alpha)+\alpha^2=e_2(\alpha)=\beta=0.
\end{equation}

\item[$(3)$] $M$ is flat, marginally trapped and it has flat normal bundle.

\item[$(4)$] The mean curvature vector field of $M$ satisfies \eqref{L2RSurfacesLemma3AH=0}.
\end{itemize}
\end{proposition}

\begin{proof}
Consider the frame field $\{e_1,e_2;\theta,\xi\}$ on $M$ defined by \eqref{LCnxRxn2DecompNew} for an $\alpha\in C^\infty(M)$. Therefore, the equations appearing on \eqref{SectLC2xRSurfsEq1All} are satisfied. Note that   \eqref{SectLC2xRSurfsEq1a} and \eqref{SectLC2xRSurfsEq1b} imply
\begin{equation}\label{L2RSurfacesLemma3MeanCurv}
H=\frac {(e_1(\alpha)+\alpha^2-\beta)}2\theta+\frac 12 \xi.
\end{equation}
By a direct computation using \eqref{SectLC2xRSurfsEq1c}, \eqref{SectLC2xRSurfsEq1d} and \eqref{L2RSurfacesLemma3MeanCurv} we obtain
\begin{equation}\label{L2RSurfacesLemma3AH}A_H=\left(
\begin{array}{cc}
 \frac{1}{2} \left(-\alpha ^2-e_1(\alpha )\right) & -\frac{1}{2} e_2(\alpha ) \\
 -\frac{1}{2} e_2(\alpha ) & -\frac{\alpha ^2}{2}+\beta-\frac{e_1(\alpha )}{2} \\
\end{array}
\right).
\end{equation}

\textit{$(1)\Rightarrow (2)$.} Assume that $M$ is  pseudo umbilical. Then, \eqref{L2RSurfacesLemma3AH} implies  $e_2(\alpha)=\beta=0$.  Furthermore, as $\beta=0$, the Codazzi equation \eqref{Codazzi} for $X=e_1$, $Y=Z=e_2$ implies  that \eqref{SecLc2xRThm1Eq1}. Hence, we have \eqref{L2RSurfacesLemma3Equa1}.

\textit{$(2)\Rightarrow (3)$.} If \eqref{L2RSurfacesLemma3Equa1} is satisfied, then Corollary \ref{LCnxRLemma1Corry2} and Lemma \ref{L2RSurfacesLemma0Flat} imply $K=K^\bot=0$. Furthermore, by combining \eqref{L2RSurfacesLemma3Equa1} with \eqref{L2RSurfacesLemma3MeanCurv} we have
$$H=\frac 12 \xi.$$
which yields that $M$ is also marginally trapped.

\textit{$(3)\Rightarrow (4)$.} If $M$ has flat normal bundle, then Corollary \ref{LCnxRLemma1Corry2} implies that $e_2(\alpha)=0$ and if $M$ is also marginally trapped, then \eqref{L2RSurfacesLemma3MeanCurv} gives
$$(e_1(\alpha)+\alpha^2-\beta)=0.$$
Moreover, if $M$ is flat, then we have \eqref{SecLc2xRThm1Eq1} by Corollary \ref{L2RSurfacesLemma0Flat}. Hence, \eqref{L2RSurfacesLemma3AH} turns into
\eqref{L2RSurfacesLemma3AH=0}.

\textit{$(4)\Rightarrow (1)$} is obvious.
\end{proof}

Now, we are ready to obtain the local classification of pseudo-umbilical surfaces lying on $\mathcal{LC}^2\times \mathbb R$.
\begin{theorem}\label{L2RSurfacesPseudoTheorem}
Let $M$ be a surface  in the Minkowski space-time $\mathbb E^4_1$ lying on $\mathcal{LC}^2\times \mathbb R.$ Then, $M$ is a  pseudo-umbilical surface if and only if for all $p\in M$ there exists a neighborhood of $M$ which is congruent to the surface parametrized by 
\begin{equation}\label{L2RSurfacesPseudoTheoremEqPosVect}
\phi(s,t)=\left(-\frac{(s+c_1) (\cos  t-3)}{2 \sqrt{2}},(s+c_1) \sin  t,\frac{(s+c_1) (3 \cos  t-1)}{2 \sqrt{2}},s\right)
\end{equation}
for a constant $c_1$.
\end{theorem}

\begin{proof}
In order to prove the necessary condition, assume that  $M$ is pseudo-umbilical and let $p\in M$. Then, by considering Proposition \ref{L2RSurfacesLemma3} and Corollary \ref{L2RIsotSurfacesLemma1}, we obtain that $M$ is flat, $\lambda$-isotropic and it has flat normal bundle. Therefore Theorem \ref{SecLc2xRThm1} and Corollary \ref{SecLc2xRCorryFNB} imply that $p$ has a neighborhood in $M$ which can be parametrized as
\begin{equation}\label{L2RSurfacesPseudoTheoremPEq1}
\phi(s,t)=\left((s+c_1)\gamma_1(t),(s+c_1)\gamma_2(t),(s+c_1)\gamma_3(t),s\right)
\end{equation}
for a constant $c_1\in \mathbb R$ and an arc length parametrized curve $\gamma$ such that $\langle\gamma,\gamma\rangle=0$. Note that \eqref{L2RSurfacesPseudoTheoremPEq1} implies that 
 \begin{equation}\label{L2RSurfacesPseudoTheoremPEq2}
\hat\alpha(s,t)=(s+c_1).
\end{equation}
 Furthermore, by Theorem \ref{L2RIsotSurfacesThm1}, the function $\kappa$ defined by \eqref{SubsectLCnEq2c} satisfies \eqref{SectLC2xRSurfsFrameEq2}. By combining \eqref{L2RSurfacesPseudoTheoremPEq2} and \eqref{SectLC2xRSurfsFrameEq2}, we obtain
$$\kappa(t)=-1/2.$$
By solving \eqref{SubsectLCnEq2b} for $\kappa(t)=-1/2$, we obtain
\begin{equation}\label{L2RSurfacesPseudoTheoremPEq3}
\gamma=\frac{\cos t+1}{2} v_1 +(1-\cos t)v_2 +\sin t v_3,\quad \eta=\frac{1-\cos t }{4} v_1 +\frac{1+\cos t }{2} v_2-\frac{\sin t }{2} v_3 
\end{equation}
for some constant vectors $v_1,v_2,v_3$. By considering that $\gamma$  is parametrized by its arc-length, we observe that $\{v_1,v_2;v_3\}$ is a pseudo-orthonormal base for $\mathbb E^3_1$ with $-\langle v_1,v_2\rangle=\langle v_3,v_3\rangle=1$. Up to a suitable isometry, we choose
$$v_1=\frac1{\sqrt 2}(1,0,1),v_2=\frac1{\sqrt 2}(1,0,-1),v_3=(0,1,0).$$
Then, by combining \eqref{L2RSurfacesPseudoTheoremPEq1} and \eqref{L2RSurfacesPseudoTheoremPEq3}, we obtain \eqref{L2RSurfacesPseudoTheoremEqPosVect}. This completes the proof of the necessary condition.

Conversely,  the surface $M$ parametrized by \eqref{L2RSurfacesPseudoTheoremEqPosVect} for a constant $c_1$. We choose the frame field formed by the vector fields
\begin{align*}
\begin{split}
e_1=&\frac{\partial }{\partial s},\qquad  e_2=\frac1{s+c_1}\frac{\partial }{\partial t},\\
\theta=&\left(-\frac{(s+c_1) (\cos  t-3)}{2 \sqrt{2}},(s+c_1) \sin  t,\frac{(s+c_1) (3 \cos  t-1)}{2 \sqrt{2}},0\right)\qquad  \\
\xi=&\left(\frac{3}{2 \sqrt{2} (s+c_1)},0,-\frac{1}{2 \sqrt{2} (s+c_1)},\frac{1}{s+c_1}\right).
\end{split}
\end{align*}
By a direct computation, we obtain
$$A_\theta=\left(
\begin{array}{cc}
 0&0\\
 0&-1
\end{array}
\right),\qquad A_\xi=0$$
which implies $H=\frac 12 \xi$. So, we have $A_H=0$. Hence, $M$ is pseudo-umbilical by Proposition \ref{L2RSurfacesLemma3}. This completes the proof of the sufficient condition.
\end{proof}

\begin{rem}\label{L2RSurfacesPseudoTheoremRem}
If $M$ is a surface is congruent to the surface parametrized by \eqref{L2RSurfacesPseudoTheoremEqPosVect} in the Minkowski space-time $\mathbb E^4_1$, then we have 
$$M\subset\mathcal{P}_d\cap\mathcal{LC}^2\times\mathbb R$$ 
for a light-like plane $\mathcal{P}_d$ with the light-like normal $n$ such that $\langle n,\partial_{z} \rangle \neq0$. Note that for the surface parametrized by \eqref{L2RSurfacesPseudoTheoremEqPosVect} we have $n=(3,0,1,2 \sqrt{2})$.
\end{rem}


\section{Conclusions}

In this paper, we studied Riemannian submanifolds lying on the light-like hypercylinder $\mathcal{LC}^n \times \mathbb{R}$ of the Lorentz-Minkowski space $\mathbb{E}_1^{n+2}$. By constructing a geometrically defined global frame field, we introduced extrinsic invariants that allowed for the characterization of several special classes of submanifolds. In particular, when $n > 2$, we showed that every $\lambda$-isotropic submanifold is necessarily pseudo-umbilical (See Corollary~\ref{L2RIsotSurfacesLemma1}), and we provided separated local classifications of both classes in Theorem~\ref{LCnxRIsotThmFORn>2} and Theorem~\ref{LCnxRProp2}, respectively. Theorem~\ref{L2RIsotSurfacesThm1} characterizes isotropic surfaces in $\mathcal{LC}^2 \times \mathbb{R}$, while Theorem~\ref{L2RSurfacesPseudoTheorem} addresses the pseudo-umbilical case. Moreover, unlike the higher-dimensional setting, when $n = 2$, the converse implication of (4) of Corollary~\ref{L2RIsotSurfacesLemma1} holds as pseudo-umbilicality implies $\lambda$-isotropy. However, in both cases, the local classification theorems show that the converse of these implications does not hold in general.

\textbf{Lying fully in the ambient space}. Finally, we would like to note that all of the submanifolds obtained in this paper appear to be contained in a hyperplane of the ambient Minkowski space (or space-time when $n=2$). However, this phenomenon holds only locally as these submanifolds are not necessarily globally contained in a single hyperplane. To illustrate this point, we construct an explicit example to the class of surfaces constructed in  Theorem \ref{L2RSurfacesPseudoTheorem} by considering the case $\mathcal{P}_d = P(O,n)$ (see Remark \ref{L2RSurfacesPseudoTheoremRem}), where $P(O,n)$ denotes the plane passing through the origin with normal vector $n = (1,0,0,1)$:

\begin{example}\label{ImplicExample}\cite{Cabrerizo-etall2010,ChenIshikawa1991} 
The surface $M = P(O,n) \cap (\mathcal{LC}^2 \times \mathbb{R})$ parametrized by
$$\phi(s,t)=\left(s,\frac{s}{\sqrt {t^2+1}},\frac{st}{\sqrt {t^2+1}},s\right)$$
is pseudo-umbilical and it lies on $\mathcal{LC}^2 \times \mathbb{R}$. Note that $\phi(s,t)$ is a re-parametrization of
$$\hat\phi(x,y)=\left(\sqrt{x^2+y^2},x,y,\sqrt{x^2+y^2}\right),$$
obtained by letting $f(x,y) = \sqrt{x^2 + y^2}$ in \cite[Example 5.1]{Cabrerizo-etall2010} (See also \cite[Example 6.1]{ChenIshikawa1991}).
\end{example}

However, the following observation should be made:
\begin{rem}
By using the same construction as in \cite[Example 5.5]{Cabrerizo-etall2010}, one can obtain a pseudo-umbilical surface in the Minkowski space-time $\mathbb{E}_1^4$ lying in $\mathcal{LC}^2 \times \mathbb{R}$ that is not globally contained in any hyperplane.
\end{rem}


\section*{Acknowledgements}

A portion of this work was derived from the Master's thesis of the first named author.

%

\section*{Declarations}

\textbf{AI Usage.} The authors used ChatGPT solely for grammar and language refinement in certain parts of the manuscript.

\textbf{Data Availability.} Data sharing not applicable to this article because no datasets were generated or analysed during the current study.

\textbf{Funding.} The second named author is supported by Scientific and Technological Research Council of T\"urkiye (T\"UBITAK).

\textbf{Code availability.} N/A.

\textbf{Conflicts of interest.} The authors have not disclosed any competing interests.

\end{document}